\documentclass[10pt,letterpaper]{article}

 \usepackage{amssymb,latexsym,amsmath,amsthm}

\usepackage{fullpage}

\newtheorem*{thm*}{Theorem A}

\newtheorem{thm}{Theorem}

\newtheorem{exam}{Example}

\newtheorem{lemma}{Lemma}

\newtheorem{remark}{Remark}

\newtheorem{cor}{Corollary}

\newtheorem{prop}{Proposition}

\author{A. Aghajani\thanks{School of Mathematics, Iran University of Science and Technology, Narmak, Tehran, Iran. Email: aghajani@iust.ac.ir.} \and
    C. Cowan\thanks{Department of Mathematics, University of Manitoba, Winnipeg, Manitoba, Canada R3T 2N2. Email: craig.cowan@umanitoba.ca. Research supported in part by NSERC.} }

\begin{document}

\def\d{ \partial_{x_j} }
\def\Na{{\mathbb{N}}}

\def\Z{{\mathbb{Z}}}

\def\IR{{\mathbb{R}}}

\newcommand{\E}[0]{ \varepsilon}

\newcommand{\la}[0]{ \lambda}

\newcommand{\s}[0]{ \mathcal{S}}

\newcommand{\AO}[1]{\| #1 \| }

\newcommand{\BO}[2]{ \left( #1 , #2 \right) }

\newcommand{\CO}[2]{ \left\langle #1 , #2 \right\rangle}

\newcommand{\R}[0]{ \IR\cup \{\infty \} }

\newcommand{\co}[1]{ #1^{\prime}}

\newcommand{\p}[0]{ p^{\prime}}

\newcommand{\m}[1]{   \mathcal{ #1 }}

\newcommand{ \W}[0]{ \mathcal{W}}

\newcommand{ \A}[1]{ \left\| #1 \right\|_H }

\newcommand{\B}[2]{ \left( #1 , #2 \right)_H }

\newcommand{\C}[2]{ \left\langle #1 , #2 \right\rangle_{  H^* , H } }

 \newcommand{\HON}[1]{ \| #1 \|_{ H^1} }

\newcommand{ \Om }{ \Omega}

\newcommand{ \pOm}{\partial \Omega}

\newcommand{\D}{ \mathcal{D} \left( \Omega \right)}

\newcommand{\DP}{ \mathcal{D}^{\prime} \left( \Omega \right)  }

\newcommand{\DPP}[2]{   \left\langle #1 , #2 \right\rangle_{  \mathcal{D}^{\prime}, \mathcal{D} }}

\newcommand{\PHH}[2]{    \left\langle #1 , #2 \right\rangle_{    \left(H^1 \right)^*  ,  H^1   }    }

\newcommand{\PHO}[2]{  \left\langle #1 , #2 \right\rangle_{  H^{-1}  , H_0^1  }}

 \newcommand{\HO}{ H^1 \left( \Omega \right)}

\newcommand{\HOO}{ H_0^1 \left( \Omega \right) }

\newcommand{\CC}{C_c^\infty\left(\Omega \right) }

\newcommand{\N}[1]{ \left\| #1\right\|_{ H_0^1  }  }

\newcommand{\IN}[2]{ \left(#1,#2\right)_{  H_0^1} }

\newcommand{\INI}[2]{ \left( #1 ,#2 \right)_ { H^1}}

\newcommand{\HH}{   H^1 \left( \Omega \right)^* }

\newcommand{\HL}{ H^{-1} \left( \Omega \right) }

\newcommand{\HS}[1]{ \| #1 \|_{H^*}}

\newcommand{\HSI}[2]{ \left( #1 , #2 \right)_{ H^*}}

\newcommand{\WO}{ W_0^{1,p}}
\newcommand{\w}[1]{ \| #1 \|_{W_0^{1,p}}}

\newcommand{\ww}{(W_0^{1,p})^*}

\newcommand{\Ov}{ \overline{\Omega}}

\title{A note on the nonexistence of positive supersolutions to elliptic equations with gradient terms}
\maketitle

\begin{abstract}

We prove that if  the elliptic problem $-\Delta u+b(x)|\nabla u|=c(x)u$ with $c\ge0$ has a positive supersolution in a domain $\Omega$ of $ \IR^{N\ge 3}$, then $c,b$ must satisfy the inequality
\[\sqrt{ \int_\Omega  c\phi^2}\le \sqrt{ \int_\Omega | \nabla\phi|^2}+\sqrt{ \int_\Omega  \frac{b^2}{4}\phi^2},~~~\phi \in C_c^\infty(\Omega).\]
As an application, we obtain Liouville type theorems for positive supersolutions in  exterior domains  when $c(x)-\frac{b^2(x)}{4}>0$ for large $|x|$, but unlike the known results we allow the case $\liminf_{|x|\rightarrow\infty}c(x)-\frac{b^2(x)}{4}=0$. Also the weights $b$ and $c$ are allowed to be unbounded. In particular, among other things, we show that if $\tau:=\limsup_{|x| \rightarrow\infty}|xb(x)|<\infty$ then this problem  does not admit any positive supersolution if
   \[\liminf_{|x| \rightarrow\infty}|x|^2c(x)> \frac{(N-2+\tau)^2}{4},\]
 and, when $\tau=\infty, $ we have the same if
\[\limsup_{R\rightarrow\infty} R\Big(\frac{ \inf_{R<|x|<2 R} (c(x)-\frac{b(x)^2}{4})}{\sup_{\frac{R}{2}<|x|<4 R}|b(x)|}\Big)=\infty.\]

\end{abstract}

\noindent
{\it \footnotesize 2010 Mathematics Subject Classification} Primary 35J60; Secondary 35B53.. {\scriptsize }\\
{\it \footnotesize Key words: Liouville type theorems; supersolutions, gradient term}. {\scriptsize }

\section{Introduction and statement of the main results}

In this work we  consider classical supersolutions of the equation
\begin{equation}\label{new}
-\Delta u+b(x)|\nabla u|=c(x)u,~~x\in\Omega,
\end{equation}
where $\Omega$ is an exterior domain $ \IR^N$, $N\ge3$.
By a classical supersolution we mean a function $u\in C^2(\Omega)$ verifying the inequality $-\Delta u+b(x)|\nabla u|\ge c(x)u$ pointwise in $\Omega$. \\
In \cite{BHN} (also see \cite{BHR}), Berestycki, Hamel and Nadirashvili implicitly
proved, as a consequence of the study of eigenvalue problems in $ \IR^N$, that if $b,c$ are continuous functions in $\IR^N$ then the problem
\begin{equation}\label{new2}
-\Delta u+b(x) \cdot \nabla u\ge c(x)u~~in~~\IR^N\
\end{equation}
does not admit any positive solution provided that $b$ and $c$ are bounded and satisfy
\begin{equation}\label{C}
\liminf_{|x|\rightarrow\infty}c(x)-\frac{b(x)^2}{4}>0.
\end{equation}
L. Rossi in \cite{R} generalized the above nonexistence results to the framework of fully nonlinear elliptic operators. As a particular case, it follows that if $b, c$ are bounded in $ \IR^N\setminus B_{R_0}$ and ($\ref{C}$) holds then Problem $\ref{new2}$ does not admit any positive solution. It's worth noting
that every  supersolution  $u$ of (\ref{new2}) is also a supersolution of (\ref{new}) as we have $b(x) \cdot \nabla u\le |b(x)||\nabla u|$. To see some related problems one can see  \cite {DC} and \cite {CF} where the authors proved some Hadamard and Liouville type properties for
nonnegative viscosity supersolutions of fully non linear uniformly elliptic partial
differential inequalities in the whole space, or in an exterior domain, for more references see \cite{AMQ1,AMQ2,AMQ3,AMQ4,ADJT,BMQ,DC,CF,CPZ,CQZ,FQS}.\\
In \cite{AMQ1} Alarcon,  Garcia-Melian and Quaas  considered positive classical supersolutions of  $(\ref{new})$ for more general unbounded weights $b$ and $c$. They proved that if $b,c\in C( \IR^N\setminus B_{R_0})$ verify $(\ref{C})$ and satisfy a further restriction related
to the fundamental solutions of the homogeneous problem (see Theorems 1.1 and 1.2 in \cite{AMQ1}) then there are no classical positive supersolutions to  $(\ref{new})$ which do (or do not) blow up at infinity. Their proof of nonexistence results depends on properties of the function $m(R) = \inf_{|x|=R} u(x)$ and  fundamental solutions of the equation $−\Delta v + \tilde{b}(|x|)|\nabla v| = 0$ in $ \IR^N \setminus B_{R_0}$, where $\tilde{b}(r):=\sup_{|x|=r}b(x)$.\\

In this paper we use a different approach, by employing  a generalized version of Hardy inequality, and obtain new Liouville type results, that seems to be sharp in some sense, and improve the results mentioned above. In particular,
 we may allow the case
\[\liminf _{|x| \rightarrow\infty}c(x)-\frac{b^2(x)}{4}=0,\]
 and without the boundedness assumption on the weights $b$ and $c$.\\

 We proceed now to the statement of our main results.
 First, using a generalized version of Hardy inequality proved by the second author in \cite{Cowan_hardy}, we prove the following lemma which is crucial in the proofs of the main results.

 \begin{prop}\label{pro-g}
If  (\ref{new}) has a solution $u>0$ then we have
 \begin{equation}\label{ineguality-general}
\sqrt{ \int_\Omega  c\phi^2}\le \sqrt{ \int_\Omega | \nabla\phi|^2}+\sqrt{ \int_\Omega  \frac{b^2}{4}\phi^2},
 \end{equation}
or equivalently,

   \begin{equation}\label{ineguality-general1}
\int_\Omega  (c-\frac{b^2}{4})\phi^2\le \int_\Omega | \nabla\phi|^2+2\sqrt{\int_\Omega | \nabla\phi|^2\int_\Omega  \frac{b^2}{4}\phi^2}
 \end{equation}
  for every $ \phi \in C_c^\infty(\Omega)$.
 \end{prop}

 The following is our first general  nonexistence result.
 \begin{thm}\label{tg}
 Let $\Omega=\IR^N \setminus B_{R_0}$ be an exterior domian and $b, c \in C( \Omega )$ with $c(x)-\frac{b^2(x)}{4}> 0$ for $|x|$ sufficiently large. Then  (\ref{new}) does not have any positive supersolution if for some   $\gamma>1$ we have

      \begin{equation}\label{g1}
\sup_{R>2R_0} \frac{\inf_{R<|x|<\gamma R}|x|^2
(c-\frac{b^2}{4})}{\sup_{\frac{R}{2}<|x|<2\gamma
R}(1+\frac{2|xb(x)| }{N-2}} > \frac{\beta^2 \ln 2+4\beta+6}{\ln
\gamma}+\beta^2,
 \end{equation}
 where $\beta:=\frac{N-2}{2}$. In particular this the case if
        \begin{equation}\label{g2}
  \frac{\liminf_{|x|\rightarrow\infty}|x|^2 (c-\frac{b^2}{4})}{\limsup_{|x|\rightarrow\infty}(1+\frac{2|xb(x)| }{N-2})} > \frac{(N-2)^2}{4},
 \end{equation}
 \end{thm}
 As a consequence we have the following result which is more easy
  to be checked in applications:
  \begin{prop}\label{pro2}
  Let $b, c \in C( \IR^N \setminus B_{R_0} )$ with $c(x)-\frac{b^2(x)}{4}> 0$ for $|x|$ sufficiently large. \\

$(i)$  Assume $\tau:=\limsup_{|x| \rightarrow\infty}|xb(x)|<\infty$. Then  (\ref{new}) does not have any positive supersolution if
       \begin{equation}\label{g3}
\liminf_{|x| \rightarrow\infty}|x|^2c(x)> \frac{(N-2+\tau)^2}{4}.
 \end{equation}
 Moreover, if $b\ge 0$ and
        \begin{equation}\label{g4}
 \limsup_{|x|\rightarrow\infty}|x|^2c(x)< \frac{(N-2+\tau)^2}{4},
 \end{equation}
then (\ref{new}) has a positive supersolution in $\IR^N \setminus B_{R_1}$ for $R_1$ sufficiently large.\\

  $(ii)$ If $\limsup_{|x| \rightarrow\infty}|xb(x)|=\infty$ then (\ref{new}) does not have any positive supersolution if
       \begin{equation}\label{simple1}
\limsup_{R\rightarrow\infty} R\Big(\frac{ \inf_{R<|x|<2 R}
(c-\frac{b^2}{4})}{\sup_{\frac{R}{2}<|x|<4 R}|b(x)|}\Big)=\infty.
 \end{equation}

 \end{prop}

 \begin{exam}

 Consider the problem
 \begin{equation} \label{exam1}
-\Delta u + b|x|^\lambda | \nabla u| \ge c|x|^\mu u \quad \IR^N\setminus B_{R_0},
\end{equation}
 where $b,c\in R$  and $\mu\ge 2\lambda$. Then it is easy to see that we have
 \[J(R):=R\Big(\frac{ \inf_{R<|x|<2 R} (c(x)-\frac{b(x)^2}{4})}{\sup_{R<|x|<2 R}|b(x)|}\Big)\ge C_0 R^{\mu-\lambda},~~when~~\mu>2\lambda,\]
 and also
 \[J(R)=C_1(c-\frac{b^2}{4})R^\lambda,~~when~~\mu=2\lambda.\]
 Then by Corollary 2 we see that (\ref{exam1})  does not admit any positive supersolution when $\mu>2\lambda$ and $c>0$, or $\mu=2\lambda$ and $c-\frac{b^2}{4}>0$. Also, in the remaining cases, its not hard to see that a positive supersolution can always be constructed for suitably large $R_0$ (see \cite{AMQ1} Corollary 1.2 ).
\end{exam}

 \begin{exam}
Consider the problem
\begin{equation} \label{exam2}
-\Delta u + |x|^\lambda | \nabla u| \ge (\frac{|x|^{2\lambda}}{4}+\frac{1}{|x|^\mu}) u, \quad  x\in\IR^N\setminus B_1,
\end{equation}
where $\lambda>-1$ and $\lambda+\mu<1$. Note here we have $c(x)-\frac{b(x)^2}{4}=\frac{1}{|x|^\mu}$,
hence, if $\mu>0$
\[\liminf_{|x|\rightarrow\infty} c-\frac{b^2}{4}=0,\]
thus none of the previous results can apply. However we have, for a fixed $\gamma>1$,
\[  \frac{\inf_{R<|x|<\gamma R}|x|^2 (c-\frac{b^2}{4})}{1+\frac{2\sup_{R<|x|<\gamma R}|xb(x)| }{N-2}} =  \frac{\inf_{R<|x|<\gamma R}|x|^{2-\mu}}{1+\frac{2\sup_{R<|x|<\gamma R}|x|^{1+\lambda} }{N-2}}  =O(R^{1-\mu-\lambda})\rightarrow\infty~~as~~R\rightarrow\infty. \]
Hence by the above result (\ref{exam2})  does not admit any positive supersolution.

\end{exam}

\begin{remark}
Note in the special case $b(x)\equiv0$, from Proposition \ref{pro2}  (with $\tau=0$ in part (i)) we see that the equation
\begin{equation}\label{b=0}
-\Delta u\ge c(x)u,
\end{equation}
in an exterior domain $\Omega$, does not admit any positive supersolution if

 \begin{equation}\label{lem-g3}
\liminf_{|x|\rightarrow\infty} |x|^2c(x)> \frac{(N-2)^2}{4}.
 \end{equation}

 It is worth noting that in \cite{DFP,P1,P2} it is pointed out that the above nonexistence result for positive supersolution to $(\ref{b=0})$ can be obtained by using Agmon-Allegretto-Piepenbrink theory \cite{Agmon}. The above is also proved  in \cite{CPZ} by a different method, where the authors for the proof employed the Kelvin transform to transfer the unbounded domain $\Omega$ into a bounded one containing the origin and then applied a result of \cite{CQZ} regarding the nonexistence of positive solutions for the problem $-\Delta u=\frac{\mu} {|x|^2}u+f,$ with Dirichlet BC. in a a bounded smooth domain $D$ containing the origin, where $0\le f\in L^\infty_{loc}(D\setminus \{0\})$ and $\mu>\frac{(N-2)^2}{4}$.\\
 However, note that our Proposition \ref{pro2} proves more, indeed by part (i) we see that (\ref{lem-g3})  is indeed sufficient for the nonexistence of positive supersolutions for the more general equation
 \begin{equation}\label{xb=0}
-\Delta u+\frac{\E(x)}{|x|}|\nabla u|\ge c(x)u,
\end{equation}
in exterior domains when $\lim_{x\rightarrow\infty}|\E(x)|=0$.
\end{remark}

As a byproduct of the above results, we can prove the following useful general estimate using a Hardy type inequality.

 \begin{cor}\label{byproduct}
 Let $E$ be a positive smooth function in an exterior domain $\Omega$ in $ \IR^N$ ($N\ge 3$) with $-\Delta E\ge0$. Then
 \begin{equation}\label{lem-g4}
\liminf_{|x|\rightarrow\infty} |x|^2\frac{-\Delta E}{E}\le  \frac{(N-2)^2}{4}.
 \end{equation}
 \end{cor}
\begin{exam} As an application of the above corollary consider  the equation
 \begin{equation}\label{lem-g5}
-\Delta u= |x|^au^p,~~in~~\Omega,
 \end{equation}
where  $a\in R$, $p>1$ and $\Omega$ is an exterior domain in $ \IR^N$ ($N\ge 3$). Now if $u$ is a positive classical supersolutions of this equation then we get, by Corollary \ref{byproduct},
\[\liminf_{|x|\rightarrow\infty}|x|^2 \frac{-\Delta u}{u}=\liminf_{|x|\rightarrow\infty} |x|^{a+2}u^{p-1}\le  \frac{(N-2)^2}{4}.\]
However, we know that a superharmonic function $u$ satisfies $u(x)\ge C |x|^{2-N}$ in $\Omega$ (see \cite{Serrin} or \cite {AS1,CM}), hence we must have $a+2+(p-1)(2-N)\le0$. Thus the above equation does not admit any positive supersolution if $p<\frac{N+a}{N-2}$,
which is a known result. Also, by a similar argument from Corollary  \ref{byproduct} we see that the equation
\[-\Delta u=\frac{\mu u}{|x|^2},~~(\mu>0)\]
does not admit any positive supersolution in an exterior domain if $\mu> \frac{(N-2)^2}{4}$.
\end{exam}

 \begin{cor}\label{cor3}
 If  (\ref{new}) has a solution $u>0$ , and there exists a smooth function $E>0$ with $-\Delta E\ge0 $ such that

   \begin{equation}\label{c1}
   b^2\le \gamma^2 \frac{|\nabla E|^2}{E^2}
    \end{equation}
    then

   \begin{equation}\label{c2}
\int_\Omega  (c-\frac{b^2}{4})\phi^2\le (1+2\gamma)\int_\Omega | \nabla\phi|^2,~~~\phi \in C_c^\infty(\Omega).
 \end{equation}
As a consequence (\ref{new}) does not have any supersolution if
      \begin{equation}\label{E}
\liminf_{|x|\rightarrow\infty} |x|^2(c-\frac{b^2}{4})>(1+2\gamma)
\frac{(N-2)^2}{4}.
 \end{equation}
 In particular, taking $E(x)=|x|^{2-N}$ we see that if $\tau:=\limsup_{|x|\rightarrow\infty}|x|b(x)<\infty$
 then (\ref{new}) does not have any supersolution if
 \[\liminf_{|x|\rightarrow\infty} |x|^2c(x)> \frac{(N-2+\tau)^2}{4}.\]
\end{cor}

\section{Proofs of the main results}

For the proof of our main results we use the following Hardy type inequality which is a special case of a result from \cite{Cowan_hardy} . For the sake of completeness we give a proof.
\begin{lemma}
Let $ E>0$ be smooth.  Then for all $ T \in \IR$ we have
\begin{equation}\label{main1}
 \int_\Omega | \nabla \phi|^2 dx \ge (T-T^2) \int \frac{| \nabla E|^2}{E^2} \phi^2 + T \int \frac{-\Delta E}{E} \phi^2 \qquad \phi \in C_c^\infty(\Omega).
 \end{equation}
 In particular, taking $T=\frac{1}{2}$ we get
 \begin{equation}\label{main 2}
 2\int_\Omega | \nabla \phi|^2 dx \ge \frac{1}{2} \int \frac{| \nabla E|^2}{E^2} \phi^2 + \int \frac{-\Delta E}{E} \phi^2 \qquad \phi \in C_c^\infty(\Omega).
 \end{equation}
\end{lemma}

\begin{proof}

Fix $ \phi \in C_c^\infty(\Omega)$ and set $v:=E^{-T} \phi$.  Then computing $ E^{2T} | \nabla v|^2$ gives
\[ E^{2T} | \nabla v|^2 = | \nabla \phi|^2 -2 T \nabla \phi \cdot \nabla E \phi  E^{-1} + \frac{T^2 | \nabla E|^2 \phi^2}{E^2},\] and now integrate this and note this term on the left is nonnegative.  Now integrating the middle term by parts (put all derivatives on $E$)  then gives the desired result.
\end{proof}

\noindent
\textbf{Proof of proposition \ref{pro-g}.} For $ t>\frac{1}{2}$  set $v:= u^\frac{1}{t}$.   Then we have
\[ \frac{-\Delta v}{v} \ge \frac{c}{t} + (t-1) \frac{| \nabla v|^2}{v^2} - b \frac{| \nabla v|}{v} \quad \mbox{ in } \Omega, \]
or,

\begin{eqnarray*}
\frac{-\Delta v}{v}+\frac{1}{2}\frac{| \nabla v|^2}{v^2}  &\ge & \frac{c}{t} + (t-\frac{1}{2}) \frac{| \nabla v|^2}{v^2} - b \frac{| \nabla v|}{v} \\
&  \ge &  \frac{c}{t}-\frac{b^2}{2(2t-1)} \quad \mbox{ in } \Omega.
\end{eqnarray*}

Recall the inequality  (\ref{main 2}),
\[
 \int_\Omega \frac{-\Delta v}{v} \phi^2 dx + \frac{1}{2} \int_\Omega \frac{ | \nabla v|^2}{v^2} \phi^2 dx \le 2 \int_\Omega | \nabla \phi|^2 dx
 \]
 for all $ \phi \in C_c^\infty(\Omega)$, then we get from above

\[\int_\Omega  c  \phi^2 dx \le 2 t \int_\Omega | \nabla \phi|^2 dx + \frac{t}{2(2t-1)} \int_\Omega b^2 \phi^2 dx.\]

Now we set
\[t=\frac{1}{2}+\frac{\sqrt{\int_\Omega b^2 \phi^2 dx}}{4\sqrt{\int_\Omega | \nabla \phi|^2 dx }}\]
to get
\[\int_\Omega  c  \phi^2 dx \le\Big(\sqrt{\int_\Omega | \nabla \phi|^2 dx }+\sqrt{\int_\Omega \frac{b^2}{4} \phi^2 dx}\Big)^2\]
or
\[\sqrt{ \int_\Omega  c\phi^2}\le \sqrt{ \int_\Omega | \nabla\phi|^2}+\sqrt{ \int_\Omega  \frac{b^2}{4}\phi^2},\]
that proves $(\ref{ineguality-general})$. Squaring both sides of  $(\ref{ineguality-general})$ gives $(\ref{ineguality-general1})$. \qed\\

\noindent
\textbf{Proof of Theorem $\ref{tg}$.} Assume (\ref{new}) has a positive supersolution $u$. Then from Proposition \ref{pro-g}, $c,b$ must satisfy inequality (\ref{ineguality-general1}).
Let $\gamma>1$, $R>2R_0 $ and take a smooth function $\psi$ in $\Omega$ with
$ \psi=0$ for $ R_0<|x|<\frac{R}{2}$ or $ |x|>2\gamma R$, $ \psi=1$ in $ R <|x|<\gamma R$, $0\le \psi\le 1$  and $\nabla \psi\le \frac{4}{R}$. Now  we consider  $\phi:=|x|^{-\beta}\psi$ as a test function in (\ref{ineguality-general1}), where $\beta:=\frac{N-2}{2}$. We have
\[\nabla \phi=-\beta |x|^{-\beta-2}\psi x+|x|^{-\beta}\nabla \psi\]
gives
\[|\nabla \phi|^2=\beta^2 |x|^{-2\beta-2}\psi^2-2\beta |x|^{-2\beta-2}\psi x.\nabla\psi+|x|^{-2\beta}|\nabla\psi|^2,\]
and then by the assumptions on $\psi$ and $\nabla\psi$ we have the estimates
\[|\nabla \phi|^2\le \beta^2 |x|^{-N}+\frac{8\beta}{R} |x|^{1-N}+ \frac{16}{R^2}|x|^{2-N},~~\frac{R}{2}<|x|< R, ~or~\gamma R<|x|<2\gamma R\]
and
\[|\nabla \phi|^2=\beta^2 |x|^{-N},~~R<|x|<\gamma R.\]
 Now we write

\[ \int_\Omega |\nabla \phi|^2 = \int_{\frac{R}{2}<|x|<2\gamma R} |\nabla \phi|^2=\int_{\frac{R}{2}<|x|<R} |\nabla \phi|^2+\int_{R<|x|<\gamma R} |\nabla \phi|^2+\int_{\gamma R<|x|<2\gamma R} |\nabla \phi|^2\]
\[:=I_1(R)+I_2(R)+I_3(R).\]
Using the fact that if $\alpha+N\not=0$ we have
\[\int_{R<|x|<T} |x|^\alpha dx=K_N\int_R^Tr^{\alpha+N-1}dr=K_N \frac{T^{\alpha+N}-R^{\alpha+N}}{\alpha+N}\]
and if $\alpha+N=0$ we have
\[\int_{R<|x|<T} |x|^\alpha dx=K_N\int_R^Tr^{\alpha+N-1}dr=K_N \ln\frac{T}{R}\]
(we set $K_N=1$ as it appears the same in both sides of the inequality) then
we compute
\[I_1(R)\le \beta^2 \ln 2+4\beta+6 :=C_{N},~~I_2(R)= \beta^2 \ln \gamma\]
and since  $I_3(R)=I_1(2\gamma R)$ we also get
\[I_3(R)\le C_{N}.\]\\
Hence, we proved that
\begin{equation}\label{grad}
 \int_\Omega |\nabla \phi|^2\le 2C_N+\beta^2 \ln \gamma.
\end{equation}
Recall from $(\ref{ineguality-general1})$ we have
\[
\int_\Omega  (c-\frac{b^2}{4})\phi^2\le \int_\Omega | \nabla\phi|^2+2\sqrt{\int_\Omega | \nabla\phi|^2\int_\Omega  \frac{b^2}{4}\phi^2}.\]
We write
\[\int_\Omega  \frac{b^2}{4}\phi^2=\int_{\frac{R}{2}<|x|<2\gamma R}  \frac{b^2}{4}\phi^2\le \frac{1}{4} \sup_{\frac{R}{2}<|x|<2\gamma R}|xb(x)|^2\int_{\frac{R}{2}<|x|<2\gamma R} \frac{\phi^2}{|x|^2},\]
and  estimate the last integral by the Hardy inequality in exterior domains (see \cite{Cowan_hardy}) as
\[\int_{\frac{R}{2}<|x|<2\gamma R} \frac{\phi^2}{|x|^2}\le (\frac{2}{N-2})^2 \int_{\Omega} |\nabla \phi|^2, \]
hence,
\[
\int_\Omega  (c-\frac{b^2}{4})\phi^2\le \Big(1+\frac{2\sup_{\frac{R}{2}<|x|<2\gamma R}|xb(x)| }{N-2}\Big)\int_\Omega | \nabla\phi|^2.\]
The above inequality together   inequality $(\ref{grad})$ give
\[
\int_\Omega  (c-\frac{b^2}{4})\phi^2 \le \Big(1+\frac{2\sup_{\frac{R}{2}<|x|<2\gamma R}|xb(x)| }{N-2}\Big) (\beta^2 \ln 2+4\beta+6+\beta^2\ln\gamma),\]
and then using the estimate
\[
\int_\Omega  (c-\frac{b^2}{4})\phi^2\ge \int_{R<|x|<\gamma R}  (c-\frac{b^2}{4})|x|^{2\beta}=\int_{R<|x|<\gamma R} |x|^{2\beta-2}|x|^2 (c-\frac{b^2}{4})\ge \inf_{R<|x|<\gamma R}|x|^2 (c-\frac{b^2}{4}) \ln\gamma,\]
we arrive at
  \begin{equation}\label{ineguality-general2}
  \frac{\inf_{R<|x|<\gamma R}|x|^2 (c-\frac{b^2}{4})}{1+\frac{2\sup_{\frac{R}{2}<|x|<2\gamma R}|xb(x)| }{N-2}} \le \frac{\beta^2 \ln 2+4\beta+6}{\ln \gamma}+\beta^2.
  \end{equation}
Hence, there exists no positive supersolution if inequality
(\ref{ineguality-general2}) violates for some $R>2R_0$ and
$\gamma>1$, that proves the first part.
Also, letting
$\gamma\rightarrow\infty$ and then $R\rightarrow\infty$ in
 (\ref{ineguality-general2}) we  see that  (\ref{new}) does
not have any positive supersolution if $(\ref{g2})$  holds.
\qed \\

\noindent
\textbf{Proof of Proposition \ref{pro2}. } To prove (i) note that from Theorem $\ref{tg}$ we have the desired result if
\[ \frac{\liminf_{|x|\rightarrow\infty}|x|^2 (c-\frac{b^2}{4})}{\limsup_{|x|\rightarrow\infty}(1+\frac{2|xb(x)| }{N-2})} > \frac{(N-2)^2}{4}.\]
Now using the fact that
\[\liminf_{|x|\rightarrow\infty}
|x|^2(c-\frac{b^2}{4})\ge \liminf_{|x|\rightarrow\infty}
|x|^2c-\frac{\tau^2}{4},\]
we see that  the  inequality above holds if
\[\liminf_{|x|\rightarrow\infty}|x|^2 c> \frac{\tau^2}{4}+\frac{(N-2)^2}{4}(1+\frac{2\tau}{N-2})=\frac{(N-2+\tau)^2}{4}.\]
To prove the second part of (i),  set
$\alpha:=\limsup_{|x|\rightarrow\infty}
|x|^2c$, then $(\ref{g4})$ reads as
$\alpha< \frac{(N-2+\tau)^2}{4}.$
Now choose $\alpha_1>\alpha$ and $\tau_1<\tau$ such that
\begin{equation}\label{sup-1}
\alpha_1\le \frac{(N-2+\tau_1)^2}{4}.
\end{equation}
Now we look for some $m>0$ so that the function $u(x)=|x|^{-m}$ is a supersolution of equation $(\ref{new})$ in $\IR^N \setminus B_{R_1}$ for $R_1$ sufficiently large. We need
\[-\Delta u+b(x)|\nabla u|-cu=|x|^{-m-2}\Big[m(N-2-m)+mb|x|-c|x|^2\Big]\ge0,\]
for $|x|$ sufficiently large. We have
\[m(N-2-m)+mb|x|-c|x|^2\ge m(N-2-m)+m\tau_1-\alpha_1=-m^2+(N-2+\tau_1)m-\alpha_1.\]
for $|x|$ sufficiently large, and note the last term is nonnegative for some $m>0$ if and only if $(\ref{sup-1})$ holds.\\
To prove part (ii), note that if inequality $(\ref{simple1})$  holds then inequality $(\ref{g1})$ holds with $\gamma=2$ and the results follows by Theorem $\ref{tg}$.\qed \\

\noindent
 \textbf{Proof of Corollary $\ref{byproduct}$.}  First note that from Lemma 1, if we take $T=1$ in $(\ref{main1})$ we then get
 \begin{equation}\label{main11}
 \int_\Omega | \nabla \phi|^2 dx \ge  \int_\Omega \frac{-\Delta E}{E} \phi^2 \qquad \phi \in C_c^\infty(\Omega).
 \end{equation}
Taking $c(x):=\frac{-\Delta E}{E}$ then from  (\ref{main11}) we see that this $c$ satisfies (\ref{ineguality-general}) with $b\equiv0$, hence from (\ref{tg}) we must have
\[\liminf_{|x|\rightarrow\infty} |x|^2c(x)\le  \frac{(N-2)^2}{4},\]
that proves  (\ref{lem-g3}). \qed \\

\noindent
\textbf{Proof of Corollary $\ref{cor3}$. }  By the assumption (\ref{c1}) we have
\[\int_\Omega \frac{b^2}{4}\phi^2\le \gamma^2 \int_\Omega \frac{|\nabla E|^2}{4E^2}\phi^2,        \]
and then by the Hardy-type inequality (see \cite{Cowan_hardy})
\[  \int | \nabla \phi|^2 \ge \int \frac{ | \nabla E|^2}{4 E^2} \phi^2 \qquad \forall \phi \in C_c^\infty(\Omega),\]
we get
\[\int_\Omega \frac{b^2}{4}\phi^2\le \gamma^2\int | \nabla \phi|^2.\]
Using this in (\ref{ineguality-general1}) we get (\ref{c2}). One can
now proceed as in the the proof of Theorem \ref{tg} to get (\ref{E}).
\qed\\


\end{document}